\documentclass[10pt]{amsart}
\usepackage[utf8]{inputenc}
\usepackage{amsthm,amssymb,amsmath,verbatim,color}
\usepackage{nccmath}
\usepackage{ulem}
\usepackage{ytableau}
\usepackage[toc,page]{appendix}
\usepackage{xcolor}
\usepackage{tikz-cd}
\usetikzlibrary{calc,decorations.pathmorphing,shapes,arrows}
\usepackage{mathtools}
\usepackage{fixltx2e}
\usepackage[all]{xy}
\usepackage{hyperref} 
\hypersetup{
    colorlinks=true,
    citecolor=black,
    linkcolor=black,
    filecolor=black,      
    urlcolor=magenta,
}

\usetikzlibrary{decorations.pathreplacing}

\newtheorem{theorem}{Theorem}[section]
\newtheorem{proposition}[theorem]{Proposition}
\newtheorem{conjecture}[theorem]{Conjecture}

\newtheorem{prop}[theorem]{Proposition}

\theoremstyle{definition}
\newtheorem{definition}[theorem]{Definition}
\newtheorem{example}[theorem]{Example}
\newtheorem{question}[theorem]{Question}

\theoremstyle{remark}
\newtheorem{remark}[theorem]{Remark}
\numberwithin{equation}{section}

\def\CC{{\mathbb C}}
\def\QQ{{\mathbb Q}}

\def\Gr{{\mathbb{G}}}
\def\LG{{\mathbb{LG}}}

\def\Hilb{{\operatorname{Hilb}}}

\def\boldx{{\mathbf{x}}}

\newcommand\qbinom[2]{ \left[ \begin{matrix} #1 \\ #2 \end{matrix} \right]_q }

\newcommand\smallqbinom[2]{ \left[ \begin{smallmatrix} #1 \\ #2 \end{smallmatrix} \right]_q }

\newcommand\qbinomprime[3]{ \left[ \begin{matrix} #1 \\ #2 \end{matrix} \right]^\prime_{#3,q} }

\newcommand\qbinomprimeprime[2]{ \left[ \begin{matrix} #1 \\ #2 \end{matrix} \right]^{\prime\prime}_{q} }

\theoremstyle{plain}
\newcommand{\thistheoremname}{}
\newtheorem{genericthm}[theorem]{\thistheoremname}


\begin{document}
\title[Filtering cohomology of ordinary and Lagrangian Grassmannians]{Filtering cohomology of ordinary and Lagrangian Grassmannians}
\author{The 2020 Polymath Jr. ``q-binomials and the Grassmannian" group$^\dagger$}\thanks{$^\dagger$ Mentor and corresponding author: Victor Reiner, School of Mathematics, University of Minnesota, Minneapolis MN 55455. Assistant Mentor: Galen Dorpalen-Barry. Team Members: Huda Ahmed, Rasiel Chishti, Yu-Cheng Chiu, Jeremy Ellis, David Fang, Michael Feigen, Jonathan Feigert, Mabel González, Dylan Harker, Jiaye Wei, Bhavna Joshi, Gandhar Kulkarni, Kapil Lad, Zhen Liu, Ma Mingyang, Lance Myers, Arjun Nigam, Tudor Popescu, Zijian Rong, Eunice Sukarto, Leonardo Mendez Villamil, Chuanyi Wang, Napoleon Wang, Ajmain Yamin, Jeffery Yu, Matthew Yu, Yuanning Zhang, Ziye Zhu, Chen Zijian
}

\maketitle
\begin{abstract}
  This paper studies, for a positive integer $m$, the subalgebra of the cohomology ring of the complex Grassmannians generated by the elements of degree at most $m$.  We build in two ways upon a conjecture for the Hilbert series of this subalgebra due to Reiner and Tudose.  The first reinterprets it in terms of the operation of $k$-conjugation,
  suggesting two conjectural bases for the subalgebras that would imply their conjecture.  The second introduces an analogous conjecture for the cohomology of Lagrangian Grassmannians. 
\end{abstract}


\section{Introduction}
\label{intro-section}

This paper concerns the Grassmannian $\Gr(\ell,\mathbb{C}^{k+\ell})$
of $\ell$-dimensional subspaces in $\CC^{k+\ell}$, and its cohomology ring with rational coefficients,
denoted here
$
R^{\ell,k}:=H^*(\Gr(\ell,\mathbb{C}^{k+\ell}),\mathbb{Q}).
$
It has the following well-known presentation following from work of A. Borel (see Glover and Homer \cite[\S 2]{GloverHomer}):
\begin{equation}
    \label{symmetric-Borel-presentation-for-Grassmannian}
    R^{\ell,k} \cong \QQ[e_1,e_2,\ldots,e_\ell,h_1,h_2,\ldots,h_k] \,\,\, / \,\,\,
    \Big( \sum_{i=0}^d (-1)^i e_i h_{d-i} \Big)_{d=1,2,\ldots, k+\ell}
\end{equation}
with $e_0=h_0=1$ and $e_i=h_j=0$ for $i \not\in\{0,1,\ldots,\ell\}$
or $j \not\in \{0,1,\ldots,k\}$.
Here $e_i, h_i$ are (up to signs)
$i^{th}$ Chern classes of the tautological $\ell$-plane and quotient
$k$-plane bundles over $\Gr(\ell,\mathbb{C}^{k+\ell})$.
The cohomology is nonvanishing only in even degrees, and we therefore find it convenient to halve the cohomological grading in considering $R^{\ell,k}$ as a graded
ring. With this grading, $\deg(e_i)=\deg(h_i)=i$.

We will be concerned with various subalgebras of $R^{\ell,k}$ and their {\it Hilbert series}.  The Hilbert series is 
defined for graded $\QQ$-vector spaces $V=\oplus_d V_d$,
by $\Hilb(V,q):=\sum_d q^d \dim_\QQ V_d$.
One has the CW-decomposition
$
\Gr(\ell,\CC^{k+\ell})=\bigsqcup_{\lambda} X_\lambda
$
into {\it Schubert cells} $X^\lambda \cong \CC^{|\lambda|}$,
indexed by partitions $\lambda=(\lambda_1,\ldots,\lambda_\ell)$
with $k \geq \lambda_1 \geq \cdots \geq \lambda_\ell \geq 0$
whose {\it Ferrers diagram} fits inside a $\ell \times k$
rectangle, that is, $\lambda \subseteq (k^\ell)$;  here $|\lambda|:=\sum_i \lambda_i$.  This implies that the
Hilbert series for $R^{\ell,k}$
is a $q$-analogue of $\binom{k+\ell}{\ell}$ called a
{\it $q$-binomial coefficient}:
$$
\Hilb(R^{\ell,k},q)=
\sum_{ \lambda \subseteq (\ell^k)} q^{|\lambda|}=:
\qbinom{k+\ell}{\ell}=\frac{[k+\ell]!_q}{[\ell]!_q \,\, [k]!_q}
$$
where $[n]!_q:=[n]_q [n-1]_q \cdots [3]_q [2]_q [1]_q$
and $[n]_q:=1+q+q^2+\cdots+q^{n-1}$.  

\subsection{The Grassmannian conjecture}
We are interested here in the
Hilbert series for certain subalgebras of $R^{\ell,k}$.
\vskip.1in
\noindent
{\bf Definition}.
{\it
For $m=0,1,2,\ldots$, let
$R^{\ell,k,m}$ denote the $\QQ$-subalgebra of $R^{\ell,k}$
generated by the homogenous elements of degree at most $m$.
}
\vskip.1in
\noindent
It is easily seen from \eqref{symmetric-Borel-presentation-for-Grassmannian} that $R^{\ell,k}$ is generated
by either $e_1,\ldots,e_\ell$ or by $h_1,\ldots,h_k$,
so that 
\begin{itemize}
    \item 
$R^{\ell,k,m}$ is also the $\QQ$-subalgebra generated by
$e_1,\ldots,e_m$, or by $h_1,\ldots,h_m$, and
\item $R^{\ell,k,m}=R^{\ell,k}$ for $m \geq \min(\ell,k)$.
\end{itemize}
Furthermore, note that the presentation~\eqref{symmetric-Borel-presentation-for-Grassmannian}
shows that this isomorphism of polynomial rings
\begin{equation}
\label{omega-on-quotients}
\begin{array}{rcl}
\QQ[e_1,\ldots,e_\ell,h_1,\ldots,h_k]
& \overset{\omega}{\longrightarrow} &
\QQ[e_1,\ldots,e_k,h_1,\ldots,h_\ell]\\
e_i &\longmapsto &h_i,\text{ for } 1 \leq i \leq \ell\\
h_j & \longmapsto &e_j,\text{ for } 1 \leq j \leq k
\end{array}
\end{equation}
induces graded ring isomorphisms
(corresponding to the homeomorphism $\Gr(\ell,\CC^{k+\ell}) \cong
\Gr(k,\CC^{k+\ell}) $)
\begin{align}
    R^{\ell,k}&\cong R^{k,\ell},\\
    R^{\ell,k,m}&\cong R^{k,\ell,m}.
\end{align}
In work on the {\it fixed point property} for $\Gr(\ell,\CC^{k+\ell})$,
O'Neill \cite{ONeill}
conjectured the form of all
graded endomorphisms of $R^{\ell,k}$.
A special case of this was proved by 
M. Hoffman \cite{Hoffman} via complicated means.
Reiner and Tudose later made
a series of successively weaker conjectures
\cite[Conj 1,2,3,4]{ReinerTudose}, any of
which would simplify Hoffman's proof.
Their strongest conjecture described the
Hilbert series for
$R^{\ell,k,m}$, using another
$q$-analogue\footnote{At $q=1$, its definition becomes the sum $\sum_{j=0}^{\ell-i} \binom{i+j-1}{j}=\binom{\ell}{i}$, sometimes called the {\it hockey stick identity}.} of the binomial coefficient 
$\binom{\ell}{i}$ for $i \geq 1$, depending also on $k$:
\begin{equation}
\label{q-binomial-prime-definition}
\qbinomprime{\ell}{i}{k}:=
\sum_{j=0}^{\ell-i} 
q^{j(k-i+1)} \qbinom{i+j-1}{j}
\end{equation}

\vskip.1in
\noindent
{\bf Conjecture. }\cite[Conj. 1]{ReinerTudose}
{\it
For each $m=0,1,2,\ldots,\min(k,\ell)$, one has
\begin{equation}
\label{main-conj}
\Hilb(R^{\ell,k,m}, q) =
1 + \sum_{i=1}^m q^i\qbinom{k}{i}  \qbinomprime{\ell}{i}{k}.
\end{equation}
}

We will refer to \eqref{main-conj} as the R-T Conjecture.
It was checked in \cite{ReinerTudose} for the extreme cases 
$m=1$ and $m=\min(\ell,k)$, but left open in all other cases.  The case of $m=\min(\ell,k)$ relied on
the following combinatorial interpretation for the summands on the right, appearing as
\cite[Prop. 8]{ReinerTudose},
and reviewed and reproven in Section~\ref{vacancy-section} below.
Given a partition $\lambda$ with $\ell(\lambda)$ nonzero parts and $\lambda_1 \leq k$, then $\lambda$ lies inside
a $\ell(\lambda) \times k$ rectangle $(k^{\ell(\lambda)})$. 
Say that $\lambda$ is {\it $i$-vacant} if $i$ is the largest integer for which the complementary {\it skew diagram}
$(k^{\ell(\lambda)})/\lambda$ contains an $i \times (i-1)$ rectangle in its southeast corner.  Then
\begin{equation}
\label{R-T-vacancy-interpretation}
 q^i \qbinom{k}{i}  \qbinomprime{\ell}{i}{k}
 =\sum_{\substack{i-\text{vacant}\\\lambda \subseteq (k^\ell)}} q^{|\lambda|}
\end{equation}

Our first contribution is a useful reformulation
of this formula.
Section~\ref{k-conjugation-section} recalls the notion of {\it $k$-bounded 
partitions} and the notion of {\it $k$-conjugation} from the theory of {\it $k$-Schur functions} \cite{LapointeLascouxMorse, k-schur-affine-schubert}.
By applying the $k$-conjugation map $\lambda \longmapsto \lambda^{\omega(k)}$ to
the $i$-vacant partitions $\lambda \subseteq (k^\ell)$
that index the summation in \eqref{R-T-vacancy-interpretation}, we
will reformuate it as follows.

\begin{theorem}
\label{k-conjugation-reformulation}
For $i=1,2,\ldots,\min(k,\ell)$, one has
$$
q^i \qbinom{k}{i}  \qbinomprime{\ell}{i}{k}
 =\sum_{\substack{\lambda: \lambda_1=i,\\
\lambda^{\omega(k)} \subseteq (k^\ell)} }
q^{|\lambda|}.
$$
\end{theorem}

Section~\ref{schur-k-schur} builds upon this formula, and conjectures the existence of two
bases for $R^{\ell,k}$ that would prove the R-T Conjecture. Recall that the 
ring of symmetric functions $\Lambda:=\QQ[h_1,h_2,\ldots]$
in infinitely many variables has a basis $\{h_\lambda\}$
indexed by partitions $\lambda$, where 
$h_\lambda:=h_{\lambda_1} h_{\lambda_2} \cdots h_{\lambda_\ell}$.
\begin{conjecture}
\label{h-basis-conjecture}
For $1 \leq m \leq \min(\ell,k)$, 
$
\{
h_\lambda: \lambda_1 \leq m \text{ and }
\lambda^{\omega(k)} \subseteq (k^\ell)
\}
$
is a $\QQ$-basis of $R^{\ell,k,m}$.
\end{conjecture}

Section~\ref{schur-k-schur} then reviews the
$k$-Schur function basis 
$\{s^{(k)}_\lambda\}_{\lambda_1 \leq k}$
for the subalgebra $\QQ[h_1,\ldots,h_k]$ of the
ring of symmetric functions $\Lambda:=\QQ[h_1,h_2,\ldots]$
in infinitely many variables.
It also reviews how the fundamental involution
$\Lambda \overset{\omega}{\rightarrow} \Lambda$
swapping $e_i \leftrightarrow h_i$ permutes the
$k$-Schur functions 
$
\omega(s_\lambda^{(k)})=s_{\lambda^{\omega(k)}}^{(k)}
$
according to the {\it $k$-conjugation} involution 
$
\lambda \leftrightarrow \lambda^{\omega(k)}.
$

\begin{conjecture}
\label{k-schur-basis-conjecture}
For $1 \leq m \leq \min(\ell,k)$, 
$
\{
s^{(\lambda_1)}_\lambda: \lambda_1 \leq m \text{ and }
\lambda^{\omega(k)} \subseteq (k^\ell)
\}
$
is a $\QQ$-basis of $R^{\ell,k,m}$.
\end{conjecture}

Section~\ref{schur-k-schur} also explains why either of Conjecture
\ref{h-basis-conjecture} or \ref{k-schur-basis-conjecture}
would imply the R-T Conjecture.

\subsection{The Lagrangian Grassmannian}
Section~\ref{Lagrangian-Grassmannian-section} concerns 
an analogue of the R-T Conjecture for the cohomology ring of the {\it Lagrangian Grassmannian} $\LG(n,\mathbb{C}^{2n})$.  Recall that this is the space of
all maximal isotropic ($n$-dimensional) subspaces of $\CC^{2n}$ endowed with a symplectic bilinear form.  Its cohomology ring $R_{\LG}^n:=H^*(\LG(n,\mathbb{C}^{2n}),\QQ)$ 
has presentation \cite[Thm. 1 with $q=0$]{KreschTamvakis-lagrangian}:
\begin{equation}
\label{Lagrangian-cohomology-presentation}
R_{\LG}^n
\cong 
\QQ[e_1,e_2,\ldots,e_n] 
\,\,\, / \,\,\,
\Big(
e_i^2 + 2\sum_{k=1}^{n-i} (-1)^k e_{i+k} e_{i-k}
\Big)_{i=1,2,\ldots,n}
\end{equation}
with $e_0:=1, e_i=0$ if $i \not\in \{0,1,\ldots,n \}$.
As with the Grassmannian, the cohomology vanishes outside
of even degrees, and after halving the grading, one has $\deg(e_i)=i$.
The Hilbert series for $R^{n}_{\LG}$ is
$$
\Hilb( R^{n}_{\LG}, q ) 
= [2]_{q} [2]_{q^2} [2]_{q^3} \cdots [2]_{q^n}
=(1+q)(1+q^2)(1+q^3) \cdots (1+q^n),
$$
a $q$-analogue of the number $2^n$. 
Let $R^{n,m}_{\LG}$ denote the $\QQ$-subalgebra of $R^{n}_{\LG}$ generated by its elements of
degree at most $m$, or equivalently, by $e_1,e_2,\ldots,e_m$.
To describe the Hilbert series of this subalgebra, we introduce 
yet another $q$-analogue of
a binomial coefficient\footnote{At $q=1$ this sums to
$\sum_{j=0}^{n-i} \binom{i+j}{i}=\binom{n+1}{i+1}$, again via the hockey-stick identity.}
$$
\qbinomprimeprime{n+1}{i+1}:=
q^i\displaystyle\sum_{j=0}^{n-i} q^{\binom{j+1}{2}}\qbinom{i+j}{i}
$$
Our second main contribution is the following conjecture.
\begin{conjecture}
\label{Lagrangian-conjecture}
For $m=1,2,\ldots,n$, one has
$$
\Hilb(R_{\LG}^{n,m},q)
=1+\sum_{\substack{1 \leq i \leq m\\ i\text{ odd}}}
\qbinomprimeprime{n+1}{i+1}.
$$
\end{conjecture}
In particular, this conjecture is consistent with the
fact that $R^{n,m}_{\LG}=R^{n,m-1}_{\LG}$ for $m$ even, since one can easily check that 
$e_m$ lies in the subalgebra
$R^{n,m-1}_{\LG}$ using the presentation \eqref{Lagrangian-cohomology-presentation}.
In Section~\ref{Lagrangian-Grassmannian-section}, we verify the two extreme cases $m=1$ and $m=n$
of Conjecture~\ref{Lagrangian-conjecture},
the latter case (Proposition~\ref{Lagrangian-filtration-q-identity})
being a $q$-analogue of the identity $2^n=\sum_{k \text{ even}} \binom{n+1}{k}$.

\section{Proof of equation \eqref{R-T-vacancy-interpretation}}
\label{vacancy-section}

In this section, we prove equation \eqref{R-T-vacancy-interpretation}. The proof is essentially a rephrasing of Reiner-Tudose's proof of the extreme case $m=\min(\ell,k)$ of the R-T Conjecture, given in \cite[Prop. 8]{ReinerTudose}.

Recall from the introduction that a
partition $\lambda=(\lambda_1,\ldots,\lambda_\ell)$
with $k \geq \lambda_1 \geq \cdots \geq \lambda_\ell \geq 0$
may be thought of as its {\it Ferrers diagram}, and
we write $\lambda \subseteq (k^\ell)$ to indicate
that its Ferrers diagram fits inside the $\ell \times k$
rectangular partition $(k^\ell)$.
If $\ell(\lambda)$ denotes the number of nonzero parts
of $\lambda$, then $\lambda$ also lies in the (possibly) smaller rectangle $(k^{\ell(\lambda)})$.  As in the Introduction, define $\lambda$ to be 
{\it $i$-vacant} if 
$i$ is the largest integer for which the complementary {\it skew partition} $(k^{\ell(\lambda)})/\lambda$ contains an $i \times (i-1)$
rectangle in its southeast corner. 

\begin{example}
Let $k=5, \ell=6$ and $\lambda=(4,4,3,3,1,0)$
with Ferrers diagram shown on the left below.
Then $\ell(\lambda)=5$ and $\lambda$ is $3$-vacant, because the complementary skew diagram $(5^5)/\lambda$ shown to its right contains a $3 \times 2$
rectangle (shaded) in its southeast corner, but not a $4 \times 3$
rectangle:

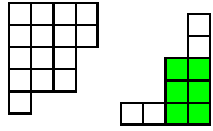
\begin{center}
\ytableausetup{smalltableaux}
\ytableausetup{centertableaux}
\ytableaushort
{} * {4,4,3,3,1,0}
\ytableaushort
{} * {4+1,4+1,3+2,3+2,1+4} * [*(green)]{0+0,0+0,3+2,3+2,3+2} 
\end{center}

\end{example}

To prove equation \eqref{R-T-vacancy-interpretation},
we must show that for $i=1,2,\ldots,\min(\ell,k)$, one has
$$
 \sum_{\substack{i-\text{vacant}\\\lambda \subseteq (k^\ell)}} q^{|\lambda|}
 =
 q^i \qbinom{k}{i}  \qbinomprime{\ell}{i}{k}
\text{ where }
\qbinomprime{\ell}{i}{k}:=
\sum_{j=0}^{\ell-i} 
q^{j(k-i+1)} \qbinom{i+j-1}{j}.
$$

\begin{proof}[Proof of \eqref{R-T-vacancy-interpretation}]
Let $\lambda \subset (k^\ell)$ be an $i$-vacant partition,
having $\ell(\lambda)$ nonzero parts. After defining $j:=\ell(\lambda)-i$, one can uniquely
decompose $\lambda$ into four subdiagrams, depicted below:

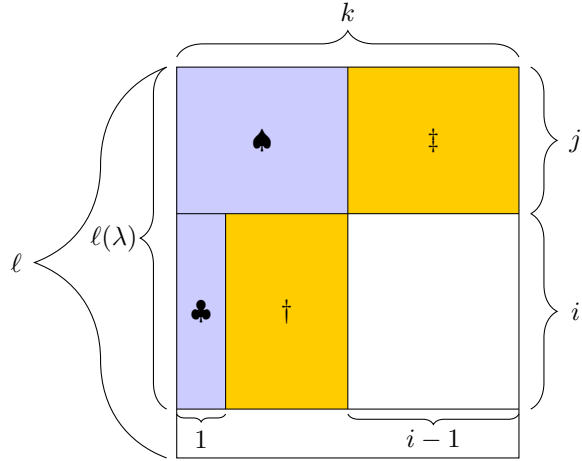
\begin{align*}
    \begin{tikzpicture}[scale=0.65]
    \fill[white] (0,0) rectangle (7, 1);
    \fill[white] (3.5, 1) rectangle (7, 5);
    \fill[red!20!yellow] (3.5, 5) rectangle (7,8);
    \fill[red!20!yellow] (1, 1) rectangle (3.5, 5);
    \fill[blue!20!white] (0, 1) rectangle (1, 5);
   \fill[blue!20!white] (0, 5) rectangle (3.5, 8);
       \draw [decorate,decoration={brace,amplitude=50pt},xshift=-3pt,yshift=0pt](-0.1,0) -- (-0.1,8.0) node [black, midway, xshift=-2.0cm]{$\ell$};
    \draw [decorate,decoration={brace,amplitude=10pt},xshift=-3pt,yshift=0pt](-0.1,1) -- (-0.1,8.0) node [black, midway, xshift=-0.7cm]{$\ell(\lambda)$};
    \draw [decorate,decoration={brace,amplitude=10pt},xshift=0pt,yshift=3pt](0,8.1) -- (7,8.1) node [black, midway, yshift=0.6cm]{$k$};
    \draw [decorate,decoration={brace,amplitude=10pt, mirror},xshift=4pt,yshift=0pt](7.1,1) -- (7.1,5) node [black, midway, xshift=0.6cm, yshift=0cm]{ $i$};
    \draw [decorate,decoration={brace,amplitude=10pt, mirror},xshift=4pt,yshift=0pt](7.1,5) -- (7.1,8) node [black, midway, xshift=0.6cm, yshift=0cm]{ $j$};
    \draw [decorate,decoration={brace,amplitude=4pt, mirror},xshift=0pt,yshift=0pt](3.5,0.9) -- (7,0.9) node [black, midway, xshift=0cm, yshift=-0.34cm]{ $i-1$};
     \draw [decorate,decoration={brace,amplitude=3pt, mirror},xshift=0pt,yshift=0pt](0,0.9) -- (1,0.9) node [black, midway, xshift=0cm, yshift=-0.34cm]{ $1$};
     \node at (0.5,3) {$\star$};
     \node at (1.75,6.5) {$\spadesuit$};
     \node at (5.25,6.5) {$\ddagger$};
     \node at (0.5,3) {$\clubsuit$};
     \node at (2.25,3) {$\dagger$};
    \draw (0,0) rectangle (7,8);
    \draw (0, 1) -- (7, 1);
    \draw (0, 5) -- (7, 5);
    \draw (3.5, 1) -- (3.5, 8);
    \draw (1, 1) -- (1, 5);
    \end{tikzpicture}
\end{align*}

\begin{itemize}
    \item A column $(1^i)$ of length $i$, denoted  ($\clubsuit$), accounting for the factor of $q^i$,
    \item a rectangle  $((k -i + 1)^{j} )$, denoted ($\spadesuit$), accounting for the factor of $q^{j(k-i+1)}$,
    \item a subpartition of the rectangle $((k - i)^i)$, denoted ($\dagger$), accounting for terms of $\smallqbinom{k}{i}$, and
    \item a subpartition of the rectangle $((i-1)^{j})$,
    denoted, ($\ddagger$), accounting for terms of $\smallqbinom{i+j-1}{j}$.
\end{itemize}
The uniqueness of this decomposition gives the proposition.
\end{proof}

\begin{example}
\label{vacant-partitions-example}
When $k=\ell=3$, the ring $R^{3,3}$ has Hilbert series
$$
\begin{aligned}
\Hilb(R^{3,3},q)
=\qbinom{6}{3}
&=1+q+2q+3q^2+3q^3+3q^4+3q^5+3q^6+2q^7+q^8+q^9\\
&=1+q \qbinom{3}{1} \qbinomprime{3}{1}{3} +
q^2 \qbinom{3}{2} \qbinomprime{3}{2}{3} +
q^3 \qbinom{3}{3} \qbinomprime{3}{3}{3}.
\end{aligned}
$$
The second term of sum in the previous line is
\begin{equation}
\label{1-vacant-term}
q \qbinom{3}{1}\qbinomprime{3}{1}{3}
=q\thinspace(1+q+q^2)\thinspace (1+q^3+q^6)
=q+q^2+q^3+q^4+q^5+q^6+q^7+q^8+q^9
\end{equation}
which is the sum of $q^{|\lambda|}$ over the $1$-vacant partitions $\lambda$ inside $(3^3)$:
$$
\ytableausetup{boxsize=0.5em}
\begin{tabular}{|c|c|c|c|c|c|c|c|c|c|}
\hline
$|\lambda|$ &1&2&3&4&5&6&7&8&9\\\hline
 & & & & & & & & & \\
$\lambda $ &
\ytableaushort{}*{1} &
\ytableaushort{}*{2} &
\ytableaushort{}*{3} &
\ytableaushort{,}*{3,1} &
\ytableaushort{,}*{3,2} &
\ytableaushort{,}*{3,3} &
\ytableaushort{,,}*{3,3,1} &
\ytableaushort{,,}*{3,3,2} &
\ytableaushort{,,}*{3,3,3}\\
 & & & & & & & & & \\
\hline
\end{tabular}
$$
The next term of the sum is
\begin{equation}
\label{2-vacant-term}
q^2 \qbinom{3}{2}\qbinomprime{3}{2}{3}
=q(1+q+q^2)(1+q^2+q^3)
=q^2+q^3+2q^4+2q^5+2q^6+q^7
\end{equation}
which is the sum of $q^{|\lambda|}$ over over the $2$-vacant partitions $\lambda$ inside $(3^3)$:
$$
\begin{tabular}{|c|c|c|c|c|c|c|}
\hline
$|\lambda|$ &2&3&4&5&6&7\\\hline
 & & & & & & \\
$\lambda $ &
\ytableaushort{,}*{1,1} &
\ytableaushort{,}*{2,1} &
\ytableaushort{,}*{2,2} &
\ytableaushort{,,}*{3,1,1} &
\ytableaushort{,,}*{3,2,1} &
\ytableaushort{,,}*{3,2,2}\\
 & & & & & &\\
  &
  &
  &
\ytableaushort{,,}*{2,1,1} &
\ytableaushort{,,}*{2,2,1} &
\ytableaushort{,,}*{2,2,2} &
\\
 & & & & & &\\
\hline
\end{tabular}
$$
The final term is
\begin{equation}
\label{3-vacant-term}
q^3 \qbinom{3}{3} \qbinomprime{3}{3}{3}
=q^3=q^{|\lambda|}
\quad \text{ for }\lambda=(1,1,1)=\ytableaushort{,,}*{1,1,1}
\end{equation}
since $\lambda=(1,1,1)$ is the unique $3$-vacant partition 
inside $(3^3)$.
\end{example}

\section{$k$-conjugation and proof of Theorem~\ref{k-conjugation-reformulation}}
\label{k-conjugation-section}

The goal of this section is to review the definition of {\it $k$-conjugation} from the theory of $k$-Schur functions, as introduced by
Lapointe, Lascoux and Morse \cite{LapointeLascouxMorse},
and use it to prove equation \eqref{k-conjugation-reformulation}.
We begin with the notions of {\it $k$-bounded partitions, $(k+1)$-cores}
and the bijection between them.

\begin{definition}
Given a partition $\lambda$, it is $k$-bounded if $\lambda_1 \leq k$.
It is a $(k+1)$-core if none of the {\it hooklengths} $h(x)$
for cells $x$ in $\lambda$ have $h(x)=k+1$;  here the
hooklength $h(x)$ for a cell $x$ in $\lambda$ is the number of
cells (including $x$ itself) which are weakly to its right in
the same row and weakly below it in the same column.
\end{definition}
Note that the definition of $(k+1)$-core 
permits boxes with  hooklength strictly greater than $k+1$.

\begin{example}
The partition $\lambda=(4,3,1,1)$ is $4$-bounded (and $5$-bounded, $6$-bounded, etc), but not $3$-bounded.  Labeling
its cells by their hook-lengths as shown here

$$
\ytableausetup{centertableaux, boxsize=1em}
\ytableaushort
{7431,521,2,1} * {4,3,1,1}
$$ 
one finds that $\lambda$ is a $6$-core (and an $8$-core, a $9$-core, etc), but not $4$-core nor a $5$-core nor a $7$-core.
\end{example}

\begin{prop}
 There is a bijective map
$\{(k+1)\text{-cores}\}\overset{p}{\to}\{k\text{-bounded partitions}\}$
that removes all boxes in a $(k+1)$-core with hook-length greater than $k+1$ and left-justifies the remaining boxes.
\end{prop}

\begin{proof}
Following \cite[Ch.2 Prop. 1.3]{k-schur-affine-schubert}, we describe the inverse map $p^{-1}$: Consider a $k$-bounded partition and work from top to bottom; for a given row, calculate the hook-lengths of its boxes; if there is a box with hook-length greater than $k$, slide this row to the right until all boxes have hook-length less than or equal to $k$. We omit the rest of the proof, which appears in \cite{k-schur-affine-schubert}.
\end{proof}

 This bijection is best described by illustrating an example labelled with hook-lengths. Considering again the $4$-bounded partition $\lambda=(4,3,1,1)$,we apply the map $p^{-1}$ to obtain a $5$-core $p^{-1}(\lambda)$:

\ytableausetup{centertableaux}
\begin{align*}
\ytableaushort
{7431,521,2,1} * {4,3,1,1} \quad\to\quad& \ytableaushort
{{11}8764321,6321,2,1} * {8,4,1,1}*[*(green)]{4,1}
\end{align*}

\begin{definition}
\label{k-conjugation-definition}
Given a $k$-bounded partition $\lambda$, we define its {\it $k$-conjugate} $\lambda^{\omega(k)}$ to be its image under the composite of these three bijections:
$$
\lambda 
\,\, \overset{p^{-1}}{\longmapsto} \,\,
p^{-1}(\lambda) 
\,\, \overset{(-)^t}{\longmapsto} \,\,
p^{-1}(\lambda)^t 
\,\, \overset{p}{\longmapsto} \,\,
p( p^{-1}(\lambda)^t ) =:
\lambda^{\omega(k)}
$$
That is, one first applies the bijection $p^{-1}$, 
then the usual {\it conjugation} or {\it transpose}
bijection $\mu \mapsto \mu^t$ that flips the Ferrers diagram across
its main diagonal,  and finally the bijection $p$.
\end{definition}

 For example, to obtain the $4$-conjugate of $\lambda=(4,3,1,1)$, we do the following operations:

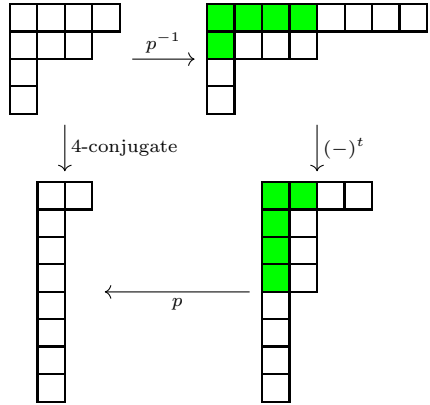
\begin{center}
\begin{tikzcd}
\ytableaushort
{} * {4,3,1,1} \arrow[r, "p^{-1}"]\arrow[d, "4\text{-conjugate}"]& \ytableaushort
{} * {8,4,1,1}*[*(green)]{4,1}\arrow[d,"(-)^t"]\\
\ytableaushort
{} * {2,1,1,1,1,1,1,1} & \ytableaushort
{} * {4,2,2,2,1,1,1,1}*[*(green)]{2,1,1,1} \arrow[l, "p"]
\end{tikzcd}
\end{center}

A glance at Theorem \ref{k-conjugation-reformulation}
shows that it follows from equation \eqref{R-T-vacancy-interpretation}
once we explain the following.
\begin{prop}\label{k-schur-interpretation}
For any $i=1,2,\ldots,\min(\ell,k)$, one has
$$
 \sum_{\substack{i-\text{vacant}\\\lambda \subseteq (k^\ell)}} q^{|\lambda|}
 =
\sum_{\substack{\mu: \mu_1=i,\\
\mu^{\omega(k)} \subseteq (k^\ell)} }
q^{|\mu|}.
$$
\end{prop}

\begin{proof}
In fact, we will show the following
\begin{quote}
{\bf Claim:} A $k$-bounded partition $\lambda$ is $i$-vacant if and only if $\mu:=\lambda^{\omega(k)}$ has $\mu_1=i$.
\end{quote}
The claim implies $k$-conjugation gives a bijection
between the indexing sets in the two sums:
\begin{itemize}
\item Given $\lambda \subseteq (k^\ell)$, then
$\lambda$ is $k$-bounded, so that $\mu:=\lambda^{\omega(k)}$ is defined.  Furthermore, $\mu_1=i$ by the claim, with $\mu^{\omega(k)}=\lambda \subseteq (k^\ell)$.
\item Given $\mu$ with $\mu_1=i$ and $\lambda:=\mu^{\omega(k)} \subseteq (k^\ell)$, then one has $\lambda$ being $i$-vacant by the claim.
\end{itemize}

To verify the claim, assume one has $\lambda \subset (k^\ell)$ which is $i$-vacant, and let $j:=\ell(\lambda)-i$. This gives the picture below inside the rectangle $(k^\ell)$, with the rectangles labeled $\heartsuit$ and $\diamondsuit$ being empty. In order for $\lambda$ to be $i$-vacant, it must contain the full rectangle $((k-i+1)^j)$ labeled $\spadesuit$, along with the column $(1^i)$ labeled $\clubsuit$, in addition to requiring that the rectangle $(i^{i-1})$ labeled $\diamondsuit$ be empty;  the rectangles labeled $\ddagger$ and $\dagger$ each contain some (possibly empty) subpartition.

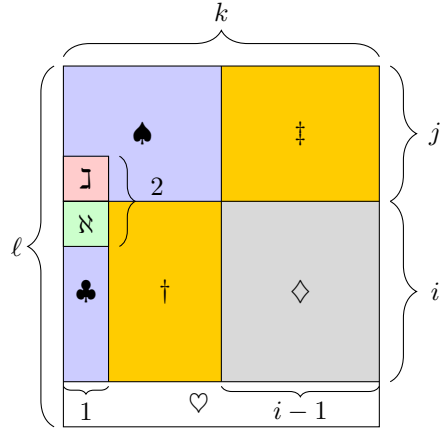
\begin{align*}
    \begin{tikzpicture}[scale=0.6]
    \fill[white] (0,0) rectangle (7, 1);
    \fill[gray!30!white] (3.5, 1) rectangle (7, 5);
    \fill[red!20!yellow] (3.5, 5) rectangle (7,8);
    \fill[red!20!yellow] (1, 1) rectangle (3.5, 5);
    \fill[red!20!white] (0,5) rectangle (1,6);
    \fill[green!20!white] (0, 4) rectangle (1, 5);
    \fill[blue!20!white] (0, 1) rectangle (1, 4);
    \fill[blue!20!white] (1, 5) rectangle (3.5, 8);
   \fill[blue!20!white] (0, 6) rectangle (1, 8);
    \draw [decorate,decoration={brace,amplitude=10pt},xshift=-3pt,yshift=0pt](-0.1,0) -- (-0.1,8.0) node [black, midway, xshift=-0.5cm]{$\ell$};
    \draw [decorate,decoration={brace,amplitude=10pt},xshift=0pt,yshift=3pt](0,8.1) -- (7,8.1) node [black, midway, yshift=0.6cm]{$k$};
    \draw [decorate,decoration={brace,amplitude=10pt, mirror},xshift=4pt,yshift=0pt](7.1,1) -- (7.1,5) node [black, midway, xshift=0.6cm, yshift=0cm]{ $i$};
    \draw [decorate,decoration={brace,amplitude=10pt, mirror},xshift=4pt,yshift=0pt](7.1,5) -- (7.1,8) node [black, midway, xshift=0.6cm, yshift=0cm]{ $j$};
    \draw [decorate,decoration={brace,amplitude=4pt, mirror},xshift=0pt,yshift=0pt](3.5,0.9) -- (7,0.9) node [black, midway, xshift=0cm, yshift=-0.34cm]{ $i-1$};
     \draw [decorate,decoration={brace,amplitude=6pt, mirror},xshift=4pt,yshift=0pt](1.1,4) -- (1.1,6) node [black, midway, xshift=0.5cm, yshift=.2cm]{ $2$};
     \draw [decorate,decoration={brace,amplitude=3pt, mirror},xshift=0pt,yshift=0pt](0,0.9) -- (1,0.9) node [black, midway, xshift=0cm, yshift=-0.34cm]{ $1$};
    \draw (0,0) rectangle (7,8);
    \draw (0, 1) -- (7, 1);
    \draw (0, 5) -- (7, 5);
    \draw (3.5, 1) -- (3.5, 8);
    \draw (1, 1) -- (1, 6);
    \draw (0, 4) rectangle (1, 5);
    \draw (0,5) rectangle (1,6);
     \node at (0.5,4.5) {$\aleph$};
      \node at (0.5,5.5) {$\gimel$};
       \node at (0.5,3) {$\clubsuit$};
     \node at (1.75,6.5) {$\spadesuit$};
     \node at (5.25,6.5) {$\ddagger$};
     \node at (2.25,3) {$\dagger$};
      \node at (5.25,3) {$\diamondsuit$};
      \node at (3,0.5) {$\heartsuit$};
    \end{tikzpicture}
\end{align*}


We want to show $\mu:=\lambda^{\omega(k)}$ has $\mu_1=i$. Note that in the process of constructing $\lambda^{\omega(k)}$, 
\begin{itemize}
\item box $\aleph$ in row $j+1$, column $1$ will not move, since its hooklength $h(\aleph) \leq (k-i+1)+i-1=k$,
\item box $\gimel$ in row $j$, column $1$ will slide right at least one box over, since 
$$h(\gimel) \geq (k-i+1)+(i+1)-1=k+1.$$
\end{itemize}
This means that $\mu=\lambda^{\omega(k)}$ has $\mu_1=i$ as claimed. 
Furthermore, if $\hat{\lambda} \subset (k^\ell)$ is assumed to be $\hat{i}$-vacant for some $\hat{i}\neq i$ then $\hat{\mu}=\hat{\lambda}^{\omega(k)}$ has $\hat{\mu}_1=\hat{i}\neq i$ as required.
\end{proof}

\begin{example}
\label{k-conjugate-filtration-example}
Continuing Example~\ref{vacant-partitions-example}, we can
reinterpret the summands in the expression
$$
\begin{aligned}
\Hilb(R^{3,3},q)
=\qbinom{6}{3}
&=1+q \qbinom{3}{1} \qbinomprime{3}{1}{3} +
q^2 \qbinom{3}{2} \qbinomprime{3}{2}{3} +
q^3 \qbinom{3}{3} \qbinomprime{3}{3}{3}.
\end{aligned}
$$
The last three terms in this expression were
interpreted before as sums of $q^{|\lambda|}$
ranging over $\lambda$ inside $(3^3)$ which are
$1$-vacant, or $2$-vacant or $3$-vacant, 
depicted in  
\eqref{1-vacant-term},
\eqref{2-vacant-term},
\eqref{3-vacant-term}.  Using Proposition~\ref{k-schur-interpretation}
they may now be reinterpreted as sums of $q^{|\lambda|}$ 
ranging over $\lambda$ with $\lambda_1=1$ or $\lambda_1=2$ or $\lambda_1=3$ whose $3$-conjugate lies inside $(3^3)$, as shown here:
$$
\ytableausetup{boxsize=0.5em}
\begin{tabular}{|c|c|c|c|c|c|c|c|c|c|}
\hline
$|\lambda|$ &1&2&3&4&5&6&7&8&9\\\hline
 & & & & & & & & & \\
$\lambda $ &
\ytableaushort{}*{1} &
\ytableaushort{,}*{1,1} &
\ytableaushort{,,}*{1,1,1} &
\ytableaushort{,,,}*{1,1,1,1} &
\ytableaushort{,,,,}*{1,1,1,1,1} &
\ytableaushort{,,,,,}*{1,1,1,1,1,1} &
\ytableaushort{,,,,,,}*{1,1,1,1,1,1,1} &
\ytableaushort{,,,,,,,}*{1,1,1,1,1,1,1,1} &
\ytableaushort{,,,,,,,,}*{1,1,1,1,1,1,1,1,1}\\
 & & & & & & & & & \\
\hline
\end{tabular}
\qquad
\begin{tabular}{|c|c|c|c|c|c|c|}
\hline
$|\lambda|$ &2&3&4&5&6&7\\\hline
 & & & & & & \\
$\lambda $ &
\ytableaushort{}*{2} &
\ytableaushort{,}*{2,1} &
\ytableaushort{,}*{2,2} &
\ytableaushort{,,,}*{2,1,1,1} &
\ytableaushort{,,,,}*{2,1,1,1,1} &
\ytableaushort{,,,,}*{2,2,1,1,1}\\
 & & & & & &\\
  &
  &
  &
\ytableaushort{,,}*{2,1,1} &
\ytableaushort{,,}*{2,2,1} &
\ytableaushort{,,,}*{2,2,1,1} &
\\
 & & & & & &\\
\hline
\end{tabular}
\qquad\ytableausetup{boxsize=0.5em}
\begin{tabular}{|c|c|}
\hline
$|\lambda|$ &3\\\hline
 &  \\
$\lambda $ &
\ytableaushort{}*{3} 
\\
 & \\
\hline
\end{tabular}
$$
\end{example}

\section{Conjectural bases and the R-T Conjecture}
\label{schur-k-schur}

Theorem~\ref{k-conjugation-reformulation} and \eqref{R-T-vacancy-interpretation}
let us rephrase the R-T Conjecture as follows: for $m=1,2,\ldots,\min(k,\ell)$,
\begin{equation}
\label{R-T-conjecture-rephrased}
\begin{aligned}
\Hilb(R^{\ell,k,m}, q) &=
1 + \sum_{i=1}^m q^i\qbinom{k}{i}  \qbinomprime{\ell}{i}{k}
\text{ where }
 q^i \qbinom{k}{i}  \qbinomprime{\ell}{i}{k}
 =
 \sum_{\substack{\lambda: \lambda_1=i,\\
\lambda^{\omega(k)} \subseteq (k^\ell)} }
q^{|\lambda|}\\
&= \sum_{\substack{\lambda: \lambda_1 \leq m,\\
\lambda^{\omega(k)} \subseteq (k^\ell)} }
q^{|\lambda|}
=\sum_{\lambda \in P^{\ell,k,m} }
q^{|\lambda|}
\end{aligned}
\end{equation}
where in the last sum we have introduced the following
indexing set of partitions:  
$$
P^{\ell,k,m}:=\{\lambda\mid \lambda_1\le m\text{, } \lambda^{\omega(k)}\subset (k^\ell)\}.
$$
This suggests a basis of $\{f_\lambda\}_{\lambda \in P^{\ell,k}}$ for $R^{\ell,k}$ consisting of
homogeneous $f_\lambda$ with $\deg(f_\lambda)=|\lambda|$,
where
\begin{equation}
    \label{big-indexing-set}
P^{\ell,k}:=P^{\ell,k,k}=\{\lambda\mid \lambda_1 \leq k\text{, } \lambda^{\omega(k)}\subset (k^\ell)\}.
\end{equation}
The R-T Conjecture is then
equivalent to the existence of such $\{f_\lambda\}_{\lambda \in P^{\ell,k}}$ with three
extra properties:
\begin{equation}
\label{desired-basis-properties}
    \begin{array}{rl}
   \text{(a)} & \text{Each }f_\lambda\text{ lies in the subalgebra }R^{\ell,k,\lambda_1}.\\
\text{(b)} &\text{The set }\{f_\lambda\}_{\lambda \in P^{\ell,k}}\text{ is a }\QQ\text{-basis for } R^{\ell,k}.\\
\text{(c)} &\text{Further, for  }m=1,2,\ldots,\min(k,\ell), \text{ the subset } \{f_\lambda\}_{\lambda \in P^{\ell,k,m}}\text{ is  a }\QQ\text{-basis for }R^{\ell,k,m}.
\end{array}
\end{equation}

\begin{remark}
In fact, properties \eqref{desired-basis-properties}(a,b) already suffice for the application to fixed point
theory that motivated the R-T Conjecture. Linear independence
of the basis elements within each subalgebra $R^{\ell,k,m}$ would show a coefficientwise
inequality
$$
\Hilb(R^{\ell,k,m}, q) \geq
1 + \sum_{i=1}^m q^i\qbinom{k}{i}  \qbinomprime{\ell}{i}{k}
$$
which suffices to prove the sequence of conjectures
\cite[Conjectures 2,3,4]{ReinerTudose}, as explained in \cite[\S 2,3]{ReinerTudose}.
\end{remark}

We next review some of the theory of symmetric functions and
$k$-Schur functions, in order to explain how it suggests
certain explicit subsets $\{f_\lambda\}_{\lambda \in P^{\ell,k}}$
with some of the properties in \eqref{desired-basis-properties}.

\subsection{The ring $R^{\ell,k}$ as a quotient of the ring of symmetric functions}
\label{classical-Schur-subsection}

The ring of symmetric functions $\Lambda_\QQ$ with $\QQ$
coefficients may be thought of as a polynomial ring
in infinitely many variables 
\begin{equation*}
\QQ[h_1,h_2,\ldots] = \Lambda_\QQ = \QQ[e_1,e_2,\ldots].
\end{equation*}
Alternatively (see, e.g., \cite[\S 7.6]{Stanley}),
it may be viewed as the quotient
$$
\Lambda_\QQ \cong \QQ[e_1,e_2,\ldots,h_1,h_2,\ldots] /
\Big(
\sum_{i=0}^d (-1)^i e_i h_{d-i}
\Big)_{d=1,2,\ldots}.
$$
Therefore, using the presentation \eqref{symmetric-Borel-presentation-for-Grassmannian}
for $R^{\ell,k}$, one can regard it as the image of
a surjection
$\Lambda_\QQ \twoheadrightarrow R^{\ell,k}$ 
in which one sets $e_i=0$ for $i > \ell$
and $h_j=0$ for $j > k$.  Note that the map $\omega: R^{\ell,k} \rightarrow R^{k,\ell}$ described in \eqref{omega-on-quotients} lifts to the {\it fundamental involution} $\omega$ on $\Lambda$, swapping $h_i \leftrightarrow e_i$.  Since the $\QQ$-subalgebra $\Lambda_\QQ^{(k)}$ of $\Lambda_\QQ$ generated
by its homogeneous elements of degree $k$ has these two descriptions
\begin{equation}
\label{k-Schur-function-ring}
\QQ[h_1,h_2,\ldots,h_k] = \Lambda^{(k)}_\QQ = \QQ[e_1,e_2,\ldots,e_k],
\end{equation}
one sees that $\omega$ restricts to an involution on $\Lambda^{(k)}_\QQ$.
Thus one has a commutative diagram, in which all of the
horizontal maps are denoted $\omega$, with both of top vertical maps inclusions, and both of the bottom vertical maps surjections:

\begin{equation}
\label{diagram-of-omegas}
\begin{tikzcd}
\Lambda_\QQ^{(k)} \arrow[r, "\omega"] \arrow[d, "i"]
& \Lambda_\QQ^{(k)} \arrow[d, "i"] \\
\Lambda_\QQ \arrow[r, "\omega"] \arrow[d, ""]
& \Lambda_\QQ \arrow[d, ""] \\
R^{k,\ell} \arrow[r, "\omega"]
& R^{\ell,k}
\end{tikzcd}
\end{equation}

\noindent
The symmetric functions $\Lambda_\QQ$ have several
important 
$\QQ$-bases indexed by all partitions $\lambda=(\lambda_1,\ldots,\lambda_\ell)$:
\begin{itemize}
    \item The {\it complete homogeneous basis} $\{h_\lambda\}$ where $h_\lambda:=h_{\lambda_1} h_{\lambda_2} \cdots h_{\lambda_\ell}$.
     \item The {\it elementary basis} $\{e_\lambda\}$ where $e_\lambda:=e_{\lambda_1} e_{\lambda_2} \cdots e_{\lambda_\ell}$.
 \item The {\it Schur functions}
$\{ s_\lambda \}$, expressible in $\{e_i\}, \{h_j\}$ 
via {\it Jacobi-Trudi formulas} \cite[\S 7.16]{Stanley}. 
\end{itemize}
The fundamental involution $\omega$ swaps $\omega: h_\lambda \leftrightarrow e_\lambda$, while permuting the Schur function basis via $\omega(s_\lambda)=s_{\lambda^t}$, where $\lambda^t$ is the {\it transpose}
or {\it conjugate} partition

The surjection $\Lambda_\QQ \twoheadrightarrow R^{\ell,k}$ has the remarkable
property \cite[\S 9.4, p. 152]{Fulton} that its kernel is the $\QQ$-span of $\{s_\lambda\}$ for which
$\lambda \not\subseteq (k^\ell)$. 
Thus $R^{\ell,k}$ has the images of 
$\{s_\lambda\}_{\lambda \subseteq (k^\ell)}$
as a $\QQ$-basis.  It then follows from unitriangularity for the expansions of $h_\lambda$ or $e_\lambda$
in the $\{s_\lambda\}$-basis \cite[eqn. (7.12.4)]{Stanley} 
that the images of
$
\{ h_\lambda \}_{\lambda \subseteq (k^\ell)}
$
and 
$
\{ e_{\lambda^t} \}_{\lambda \subseteq (k^\ell)}
=
\{ e_\lambda\}_{\lambda \subseteq (\ell^k)}
$
also give $\QQ$-bases for $R^{\ell,k}$.

\subsection{The $h_\lambda$ basis conjecture}
 Note that the basis elements $h_\lambda$ are homogeneous of degree $|\lambda|$,
and since $\lambda_1 \geq \cdots \geq \lambda_\ell$, the image of $h_\lambda$
in the quotient $R^{\ell,k}$ lies within the subalgebra $R^{\ell,k,\lambda_1}$.
Thus the images of $\{h_\lambda\}_{\lambda \in P^{\ell,k}}$
satisfy condition \eqref{desired-basis-properties}(a) above.  The following is then a rephrasing of
Conjecture~\ref{h-basis-conjecture} from the Introduction.

\vskip.1in
\noindent
{\bf Conjecture~\ref{h-basis-conjecture}.}
{\it
The images of $\{h_\lambda\}_{\lambda \in P^{\ell,k}}$ in $R^{\ell,k}$
satisfy the three conditions \eqref{desired-basis-properties} (a,b,c), giving a basis for $R^{\ell,k}$
that proves the R-T Conjecture.
}
\vskip.1in

The conjectural basis in Conjecture~\ref{h-basis-conjecture}
should be compared with a more obvious basis of
elements $\{h_\lambda\}$ for each $R^{\ell,k,m}$
that we explain here.
Abbreviate the ring presentation
\eqref{symmetric-Borel-presentation-for-Grassmannian} as 
$
R^{\ell,k} 
=\QQ[\mathbf{e},\mathbf{h}]/I^{\ell,k}.
$
One can choose the {\it lexicographic order} 
on the monomials in $\QQ[\mathbf{e},\mathbf{h}]$,
in which 
$$
e_k > \cdots > e_2 > e_1 >h_\ell > \cdots > h_2 > h_1.
$$
One can then compute the
{\it lex Gr\"obner basis} for the ideal
$I^{\ell,k}$, along with the set of accompanying
standard monomials, that is, the 
monomials that do not
divide the leading term of any Gr\"obner basis
element.  The theory of {\it Elimination ideals} \cite[Chap. 3]{CoxLittleOShea} implies that for each 
$m=1,2,\ldots,\min(k,\ell)$, the standard monomials
$\{h_\lambda\}$ only involving the
variables $h_1,\ldots,h_m$ will give a $\QQ$-basis
for $R^{\ell,k,m}$.

Unfortunately, the task of
explicitly describing in general
these lex-standard monomial bases $\{h_\lambda\}$
for the subalgebras $R^{\ell,k,m}$, and matching
their degrees to the prediction of the R-T Conjecture, has eluded us.  See the REU report \cite{REUreport} for conjectures in this regard when $\min(k,\ell) \leq 3$.  

This lex-standard monomial basis $\{h_\lambda\}$ generally 
differs from the one in Conjecture~\ref{h-basis-conjecture}.

\begin{example}
When $\ell=k=3$, the lex Gr\"obner basis
with $e_3 > e_2 > e_1> h_3 > h_2 > h_1$
can be computed using {\tt Macaulay2},
and the 
initial terms of its elements are
$
\{
h_1^{10}, h_2 h_1^6, h_2^2 h_1^3, h_2^3, h_3 h_1, h_3 h_2, h_3^2
\}.
$
Consequently the lex standard monomial basis elements
for $R^{\ell,k}$ will be as follows:
\begin{center}
\begin{tabular}{|c|c|c|c|c|c|c|c|c|c|c|}\hline
degree &0&1&2&3&4&5&6&7&8&9\\\hline\hline
$R^{\ell,k,0}$ &1&&&&&&&&&\\\hline
$R^{\ell,k,1}/R^{\ell,k,0}$ & &$h_{1}$&$h_{11}$&$h_{111}$&$h_{1111}$&$h_{11111}$&$h_{111111}$&$h_{1111111}$&$h_{11111111}$&$h_{111111111}$\\\hline
$R^{\ell,k,2}/R^{\ell,k,1}$ & & &$h_{2}$&$h_{21}$&$h_{211}$&$h_{2111}$&$h_{21111}$&$\mathbf{h_{211111}}$& & \\
 & & & & &$h_{22}$&$h_{221}$&$h_{2211}$& & & \\\hline
$R^{\ell,k,3}/R^{\ell,k,2}$ & & & &$h_{3}$& & & & & & \\\hline
\end{tabular}
\end{center}
Comparing with Example~\ref{k-conjugate-filtration-example},
one has Conjecture~\ref{h-basis-conjecture}
predicting a basis differing only in degree $7$, replacing the element $h_{211111}$ (shown in boldface above) with
the element $h_{22111}$.
\end{example}

\subsection{The $k$-Schur function basis conjecture}
As noted in the Introduction, the ring $R^{\ell,k}$
is generated by the images of 
$\{h_1,h_2,\ldots,h_k\}$ or $\{e_1,e_2,\ldots,e_k\}$, 
so that it can be viewed
as a quotient of the polynomial ring
$\Lambda_\QQ^{(k)}$ from \eqref{k-Schur-function-ring}
using the composite of the vertical maps
on either side of 
\eqref{diagram-of-omegas}.
Letting $P^{k}$ denote the set of $k$-bounded partitions $\lambda$, then this set 
indexes two obvious $\QQ$-bases for $\Lambda_\QQ^{(k)}$ which are exchanged by $\omega$, namely 
$\{ h_\lambda \}_{\lambda \in P^k}$ 
and $\{ e_\lambda \}_{\lambda \in P^k}$.

On the other hand, motivated by {\it Macdonald's positivity conjecture}, Lapointe, Lascoux and Morse \cite{LapointeLascouxMorse} introduced
the  $k$-Schur function basis 
$\{ s_\lambda^{(k)} \}_{\lambda \in P^k}$ for  $\Lambda_\QQ^{(k)}$,
playing a role analogous to the 
Schur function basis 
$\{s_\lambda\}$ for $\Lambda_\QQ$.  There
are several ways to define the
$k$-Schur functions, with an additional parameter $t$.
Here will use their ``parameterless" specialization
to $t=1$, which has the following unitriangular
relation  \cite[Property 28]{LapointeMorse} to the Schur functions:
\begin{equation}
\label{Schur-unitriangularity}
    s_\lambda^{(k)} = s_\lambda +\sum_{\substack{\mu \in P^k:\\\mu\triangleright\lambda}}
    d_{\lambda\mu}^{(k)}s_\mu
\end{equation}
where $\mu\triangleright\lambda$ is the {\it dominance partial ordering} \cite[\S 7.2]{Stanley} on partitions of a fixed size $n$, and $d_{\lambda\mu}^{(k)}$ are integer coefficients that we will not
describe or use here.  

The $k$-conjugation operation  
on the set $P^k$ of
$k$-bounded partitions from Definition~\ref{k-conjugation-definition}  was also introduced by Lapointe, Lascoux and Morse \cite{LapointeLascouxMorse}, 
to (conjecturally) describe the action of $\omega$ on $\Lambda_\QQ^{(k)}$ in this $k$-Schur function basis, and proven later by
Lapointe and Morse \cite[Theorem 38]{LapointeMorse}:

\begin{equation}
\label{involution-k-conj}
\omega( s_\lambda^{(k)} )
=s_{\lambda^{\omega(k)}}^{(k)}
\end{equation}

Since the images of 
$\{s_\lambda\}_{ \lambda \subseteq (k^\ell)}$
form a $\QQ$-basis for $R^{\ell,k}$, one can check that
\eqref{Schur-unitriangularity} implies
that the images of 
$\{s^{(k)}_\lambda\}_{ \lambda \subseteq (k^\ell)}$
also form a $\QQ$-basis for $R^{\ell,k}$.  The latter basis may also
be indexed by the set $P^{\ell,k}$ from
\eqref{big-indexing-set}, since 
$
s^{(k)}_\lambda=s^{(k)}_{\mu^{\omega(k)}}
$
where $\mu=\lambda^{\omega(k)}$,
and $\lambda$ lies in $(k^\ell)$ if and only if
$\mu$ lies in $P^{\ell,k}$.  While this basis 
$
\{
s^{(k)}_{\mu^{\omega(k)}}
\}_{\mu \in P^{\ell,k}}
$
satisfies property (4.3) (b), it fails property (4.3) (a), since the image of $s^{(k)}_{\mu^{\omega(k)}}$
 within $R^{\ell,k}$
need not lie in the subalgebra $R^{\ell,k,\mu_1}$. 
Instead
we posit the following, a rephrasing of Conjecture~\ref{k-schur-basis-conjecture}
from the Introduction.
\vskip.1in
\noindent
{\bf Conjecture~\ref{k-schur-basis-conjecture}.}
{\it
The images of $\{s^{(\lambda_1)}_\lambda\}_{\lambda \in P^{\ell,k}}$ in $R^{\ell,k}$
satisfy the three conditions \eqref{desired-basis-properties} (a,b,c), giving a basis for $R^{\ell,k}$ that proves the R-T Conjecture.
}
\vskip.1in

\section{Extreme cases for the Lagrangian Grassmannian 
Conjecture~\ref{Lagrangian-conjecture}}
\label{Lagrangian-Grassmannian-section}

Recall from the Introduction that the
Lagrangian Grassmannian $\LG(n,\mathbb{C}^{2n})$
is the space of
all maximal isotropic ($n$-dimensional) subspaces of $\CC^{2n}$ endowed with a symplectic bilinear form.  A presentation for its 
cohomology ring $R_{\LG}^n:=H^*(\LG(n,\mathbb{C}^{2n}),\QQ)$ was
given in \eqref{Lagrangian-cohomology-presentation}.

There is a CW-decomposition 
$\LG(n,\mathbb{C}^{2n})=\bigsqcup_\lambda X^\lambda_{\LG}$
into Schubert cells $X^\lambda_{\LG} \cong \CC^{|\lambda|}$
indexed by {\it strict partitions} 
$\lambda=(\lambda_1,\lambda_2,\ldots,\lambda_\ell)$
having 
$
n \geq \lambda_1 > \lambda_2 > \cdots > \lambda_\ell >0.
$
We will identify such strict partitions $\lambda$ with their {\it shifted Ferrers diagram}, 
having $\lambda_i$ boxes in row $i$, with each row shifted one unit right of the previous one.
Note that the condition $\lambda_1 \leq n$ implies that these $\lambda$ will lie inside
an {\it ambient triangle} $\Delta_n:=(n,n-1,\ldots,2,1,0)$. 

\begin{example}
Fix $n=5$. Shaded below is the shifted diagram for $\lambda= (4,2,1)$, inside the ambient triangle $\Delta_5=(5,4,3,2,1)$:
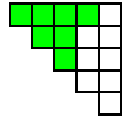
\begin{center}
\ytableausetup{smalltableaux}
\ytableausetup{centertableaux}
\ytableaushort
{} * {5,1+4,2+3,3+2,4+1} * [*(green)]{4,1+2,2+1}
\end{center}
\end{example}

\noindent
As a consequence of the Schubert cell decomposition,
the cohomology ring $R^{n}_{\LG}$ has Hilbert series

\begin{equation}
    \label{Lagrangian-hilb-repreated}
\Hilb( R^{n}_{\LG}, q ) 
=\sum_{\lambda \subseteq \Delta_n} q^{|\lambda|}
=(1+q)(1+q^2)(1+q^3) \cdots (1+q^n)
= [2]_{q} [2]_{q^2} [2]_{q^3} \cdots [2]_{q^n},
\end{equation}
a $q$-analogue of the number $2^n$.  

Define
$R^{n,m}_{\LG}$ to be the $\QQ$-subalgebra of $R^{n}_{\LG}$ generated by its 
elements of degree at most $m$, or equivalently, by the images of 
$e_1,e_2,\ldots,e_m$.
Then Conjecture~\ref{Lagrangian-conjecture}, which is suggested by extensive
computational evidence,
asserts the following:
for $m=1,2,\ldots,n$, one has
\begin{equation}
    \label{Lagrangian-conjecture-repeated}
\Hilb(R_{\LG}^{n,m},q)
=1+\sum_{\substack{1 \leq i \leq m\\ i\text{ odd}}}
\qbinomprimeprime{n+1}{i+1}
\end{equation}
where recall that we defined the following $q$-analogue of $\binom{n+1}{i+1}$:
\begin{equation}
\label{binomial-prime-prime-repeated}
\qbinomprimeprime{n+1}{i+1}:=
q^i\displaystyle\sum_{j=0}^{n-i} q^{\binom{j+1}{2}}\qbinom{i+j}{i}
\end{equation}
The next two subsections verify the two extreme cases $m=1$ and $m=n$ of this conjecture.

\subsection{The $m=1$ case of Conjecture~\ref{Lagrangian-conjecture}}

Verifying the $m=1$ case of the conjecture is the analogue
of \cite[Thm. 6, Rem. 7]{ReinerTudose} for the Grassmannian.
The $m=1$ case of \eqref{Lagrangian-conjecture-repeated} asserts
$$
\begin{aligned}
\Hilb(R_{\LG}^{n,1},q)
=1+ \qbinomprimeprime{n+1}{2} 
&=1+q^1\displaystyle\sum_{j=0}^{n-1} q^{\binom{j+1}{2}}\qbinom{1+j}{1}
=1+q \sum_{j=0}^{n-1} q^{\binom{j+1}{2}} [j+1]_q\\
&=1+q(q^0[1]_q + q^1 [2]_q+q^3[3]_q +q^6[4]_q + \cdots + q^{\binom{n}{2}}[n]_q) \\
&=1+q+q^2+q^3+q^4+\cdots+q^{\binom{n+1}{2}}
\end{aligned}
$$
Since \eqref{Lagrangian-hilb-repreated} shows that $R^n_\LG$ vanishes in degrees beyond $d:=\binom{n+1}{2}=\dim \LG(n,\mathbb{C}^{2n})$, the same holds for its subalgebra 
$R_{\LG}^{n,1}$ generated by $e_1$.  Thus the $m=1$ case of the conjecture is equivalent to
the nonvanishing assertion that the $d^{th}$ power $e_1^d \neq 0$ in $R_{\LG}$. 
There are two ways one can deduce this.
\begin{itemize}
 \item Since $R^n_{\LG} \cong H^*(X,\QQ)$
where $X=\LG(n,\mathbb{C}^{2n})$ is a smooth complex projective variety of dimension $d$, the nonvanishing
follows from the Hard Lefschetz Theorem \cite[p. 122]{GriffithsHarris}.
 \item Iterating a {\it Pieri formula} (see, e.g., \cite[eqn. (51)]{KreschTamvakis-lagrangian}) for $R^n_{\LG}$ shows that $e_1^d$ is a nonzero scalar multiple of the orientation class of $\LG(n,\mathbb{C}^{2n})$. This 
 multiple is the {\it degree} of $\LG(n,\CC^{2n})$ in its Pl\"ucker embedding, computed several times before, e.g., by Hiep and Tu \cite[\S 5 Prop. 2]{HiepTu}.
\end{itemize}

\subsection{The $m=n$ case of Conjecture~\ref{Lagrangian-conjecture}}

The other extreme case $m=n$ comes from the following proposition,
playing a role for $\LG(n,\CC^{2n})$ analogous to that of 
Proposition~\ref{k-schur-interpretation} for $\Gr(\ell,\CC^{k+\ell})$.

\begin{prop}
\label{Lagrangian-filtration-q-identity}
For $n \geq 0$, one has
$$
\sum_{\lambda \subseteq \Delta_n} q^{|\lambda|}
=1+\sum_{\substack{1 \leq i \leq n\\ i\text{ odd}}}
\qbinomprimeprime{n+1}{i+1}.
$$
\end{prop}
\begin{proof}
When $\lambda \subseteq \Delta_n$ is non-empty, 
we will uniquely define a triple $(i,j,\mu)$, where 
\begin{itemize}
    \item[(a)] $i$ is odd in the range $1 \leq i \leq n$, 
    corresponding to $i$ squares at the end of the first row of $\lambda$, darkly shaded below, and
    accounting for the $q^i$ on the right in
    \eqref{binomial-prime-prime-repeated},
    \item[(b)] $j$ is in the range $0 \leq n-i$, 
    corresponding to a subtriangle $\Delta_j \subseteq \lambda$,
    lightly shaded below, and
     accounting for the $q^{\binom{j+1}{2}}$ on the right in
    \eqref{binomial-prime-prime-repeated},
    \item[(c)] $\mu$ is the rest of $\lambda$, fitting inside a $j \times i$ rectangle $(i^j)$,
    below the $i$ first row squares from (a),
    depicted in white below,
     and accounting for the terms in $\qbinom{i+j}{i}$ on the right in
    \eqref{binomial-prime-prime-repeated}. 
\end{itemize}
To define $(i,j,\mu)$, given $\lambda$ having $\ell(\lambda)$ nonzero parts,
let $j$ be the maximum value such that 
$\Delta_j\subseteq\lambda$ and $i:=\lambda_1-j$ is {\it odd}. 
Equivalently, this means $j$ is defined by two cases:
\begin{equation}
\label{parity-cases}
j=
\begin{cases}
\ell(\lambda) & \text{, if } \lambda_1-\ell(\lambda) \text{ is odd} \\
\ell(\lambda) - 1 & \text{, if } \lambda_1-\ell(\lambda) \text{ is even}.
\end{cases}
\end{equation}
A few examples of the resulting decomposition of $\lambda$ into its parts (a),(b),(c) are illustrated here. 
In both (i) and (ii), $\lambda$ has $i=3$ and $j=4$, so that $\mu$ sits inside $4 \times 3$ rectangle $(3^4)$, but one has different lengths $\ell(\lambda)$.  In (iii), $i=1, j=5$ and $\mu$ fits 
inside a $5 \times 1$ rectangle $(1^5)$.

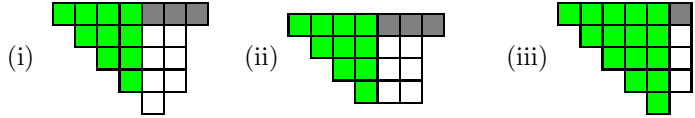
\begin{center}
\ytableausetup{centertableaux}
(i)$\ $
\ytableaushort
{} * {7,1+5,2+4,3+3,4+1} * [*(green)]{4,1+3,2+2,3+1} * [*(gray)]{4+3} 
$\quad$
(ii)
\ytableaushort
{} * {7,1+5,2+4,3+3} * [*(green)]{4,1+3,2+2,3+1} * [*(gray)]{4+3}
$\qquad$
(iii)
\ytableaushort
{} * {6,1+5,2+4,3+3,4+1} * [*(green)]{5,1+4,2+3,3+2,4+1} * [*(gray)]{5+1} 
\end{center}
Conversely, given any triple $(i,j,\mu)$ having $i$ odd in the range $1 \leq i \leq n$,
with $j$ in the range $0 \leq j \leq n-i$, and $\mu \subseteq (i^j)$, one can construct a corresponding $\lambda \subseteq \Delta_n$ having this triple as its parameters, by assembling its diagram from 
a triangle $\Delta_j$ together with $i$ more squares in the first row, placing the diagram of $\mu$ just below those $i$ square in the first row.  The two cases
in \eqref{parity-cases} correspond to whether
$
\ell(\mu)< \ell(\lambda)
$
or
$
\ell(\mu)= \ell(\lambda),
$
respectively.  This proves the proposition.
\end{proof}

\begin{question}
Is there a Lagrangian or shifted analogue of Theorem~\ref{k-conjugation-reformulation}, involving
some notion of {\it $k$-conjugation for shifted Young diagrams}?
Could this be related to an analogue of $k$-Schur functions
$s^{(k)}_\lambda$ for Schur's $Q$-functions $Q_\lambda$ in the Schubert calculus of $\LG(n,\CC^{2n})$?
\end{question}

\begin{remark} 
Related to the Lagrangian Grassmannian is
the Orthogonal Grassmannian, discussed for example in \cite{KreschTamvakis-orthogonal}.  One might ask for an analogue
of Conjecture~\ref{Lagrangian-conjecture} for its
cohomology ring $R^n_{\mathbb{OG}}$. We explain here why
the conjecture is the same. The
rings $R^n_{\mathbb{OG}}$ has a presentation (see, e.g., \cite[Thm. 1]{KreschTamvakis-orthogonal} with $q=0$) similar to that
of $R^n_{\LG}$.   In fact there is a graded
ring isomorphism $R^n_{\mathbb{OG}} \cong R^n_{\LG}$
stemming from the work of Borel, who showed that
both rings have a description (see, e.g., \cite[\S 3]{ReinerWooYong})
\begin{equation}
\label{lagrangian-orthogonal-isomorphism}
R^n_{\mathbb{OG}} \cong 
\left(
\mathbb{Q}[\boldx]/(\QQ[\boldx]^{B_n}_+)
\right)^{S_n}
\cong R^n_{\LG}
\end{equation}
as
the {\it $S_n$-invariant subalgebra}
of the {\it coinvariant algebra} $\mathbb{Q}[\boldx]/(\QQ[\boldx]^{B_n}_+$,
where $B_n$ is the {\it hyperoctahedral group} of $n \times n$ signed permutation matrices
acting on the polynomial ring $\QQ[\boldx]:=\QQ[x_1,\ldots,x_n]$, and the $S_n$ the parabolic subgroup of all permutation matrices inside $B_n$.
This is because the the Lagrangian and Orthogonal Grassmannians are homogeneous spaces for the groups $G=Sp_{2n}, SO_{2n+1}$,
which share $B_n$ as their Weyl group,
and are both quotients $G/P$ by parabolic subgroups $P$ that
correspond to the symmetric group $S_n$ inside $B_n$.  The
graded ring isomorphism \eqref{lagrangian-orthogonal-isomorphism} implies that
the subalgebras $R^{n,m}_{\mathbb{OG}}, R^{n,m}_{\LG}$ inside $R^n_{\mathbb{OG}}, R^n_{\LG}$ 
generated by the homogeneous
elements of degree at most $m$ will also be isomorphic as graded algebras, and have
the same Hilbert series.

\end{remark}



\section*{Acknowledgements}
Work of mentor and co-mentor supported by NSF grant DMS-1601961. The team also sincerely thanks the Polymath Jr.~REU faculty organizers 
(Kira Adaricheva, Ben Brubaker, Pat Devlin, Steven Miller,
Alexandra Seceleanu, Adam Sheffer, Yunus Zeytuncu)
for their vision and hard work in making this new bold new opportunity a reality.



\begin{thebibliography}{10}


\bibitem{CoxLittleOShea}
David A. Cox, John Little and Donal O'Shea.
Ideals, varieties, and algorithms.
An introduction to computational algebraic geometry and commutative algebra. Fourth edition. {\it Undergraduate Texts in Mathematics}. 
Springer, 2015.

\bibitem{Fulton}
William Fulton,
Young tableaux,
{\it London Mathematical Society Student Texts} {\bf 35},
Cambridge University Press, Cambridge, 1997.

\bibitem{GloverHomer}
Henry Glover and William D. Homer,
Endomorphisms of the cohomology ring of finite Grassmann manifolds. 
Geometric applications of homotopy theory (Proc. Conf., Evanston, Ill., 1977), I,170--193,
{\it Lecture Notes in Math.} {\bf 657}. Springer, Berlin, 1978.

\bibitem{GriffithsHarris}
Phillip Griffiths and Joe Harris, Principles of algebraic geometry. 
Wiley Classics Library. John Wiley \& Sons, Inc.,
New York, 1994.

\bibitem{HiepTu}
Dang T. Hiep and NguyenC. Tu,
An identity involving symmetric polynomials and the geometry of Lagrangian Grassmannians.
{\it J. Algebra} {\bf 565} (2021), 564--581.


\bibitem{Hoffman} Michael Hoffman, Endomorphisms of the cohomology of complex Grassmannians. {\it Trans. Amer. Math. Soc.} {\bf 281} (1984), 745--760.







\bibitem{KreschTamvakis-orthogonal}
Andrew Kresch and Harry Tamvakis,
Quantum cohomology of orthogonal Grassmannians. 
{\it Compos. Math.} {\bf 140} (2004), no. 2, 482–500. 

\bibitem{KreschTamvakis-lagrangian}
Andrew Kresch and Harry Tamvakis,
Quantum cohomology of the Lagrangian Grassmannian. 
{\it J. Algebraic Geom.} {\bf 12} (2003), no. 4, 777–810.

\bibitem{k-schur-affine-schubert}
Thomas Lam, Luc Lapointe, Jennifer Morse, Anne Schilling, Mark Shimozono, Mike Zabrocki, $k$-Schur Functions and Affine Schubert Calculus. {\it Fields Institute Monographs} {\bf 33}, 2014.

\bibitem{LapointeLascouxMorse}
Luc Lapointe, Alain Lascoux and Jennifer Morse. Tableau atoms and a new Macdonald positivity conjecture. 
{\it Duke Math. J.} {\bf 116} (2003), 103--146, 


\bibitem{LapointeMorse}
Luc Lapointe, Jennifer Morse, A $k$-Tableau Characterization of $k$-Schur Functions. 
{\it Adv. Math.} {\bf 213} (2007), 183--204.

\bibitem{ONeill} Larkin S. O'Neill, The fixed point property for Grassmann's manifold. Ph.D. Thesis, Ohio State Univ., Columbus, Ohio, 1974

\bibitem{REUreport}
The 2020 Polymath Jr. REU {\it q-Binomials and the Grassmannian} group,
Filtering Grassmannian cohomology via $k$-Schur functions, REU report, 2020; available at \url{http://www-users.math.umn.edu/~reiner/REU/REU2020notes/q_Binomials_and_the_Grassmannian.pdf}.



\bibitem{ReinerTudose}
Victor Reiner and Geanina Tudose,
Conjectures on the cohomology of the Grassmannian.
preprint 2003. 
{\tt arXiv:math/0309281}.

\bibitem{ReinerWooYong}
Victor Reiner, Alexander Woo and Alexander Yong,
Presenting the cohomology of a Schubert variety,
{\it Trans. Amer. Math. Soc.} {\bf 363} (2011), 521--543. 
.


\bibitem{Stanley}
Richard P. Stanley. 
Enumerative Combinatorics, Vol 2. 
Cambridge University Press, 2003. 





\end{thebibliography}
\end{document}